\documentclass[11pt,a4paper, twoside, leqno, final]{amsart}
\usepackage[english]{babel}
\usepackage[utf8]{inputenc}
\usepackage{fancyhdr}
\usepackage{indentfirst}
\usepackage{graphicx}
\usepackage{newlfont}
\usepackage[active]{srcltx}
\usepackage{mathrsfs}
\usepackage{amssymb}
\usepackage{amsmath}
\usepackage{latexsym}
\usepackage{amsthm}
\usepackage[all]{xy}
\usepackage{booktabs}
\usepackage{multicol}
\usepackage[dvipsnames,usenames]{color}
\usepackage[refpage,prefix]{nomencl}
\usepackage{comandi}
\usepackage[a4paper,left=3cm,right=3cm,top=4cm,bottom=4cm]{
geometry }

\numberwithin{equation}{section}

\theoremstyle{plain}  
\newtheorem{thm}{Theorem}[section]                

\newtheorem{cor}[thm]{Corollary}        
\newtheorem{lem}[thm]{Lemma}            

\newtheorem*{thma}{Theorem A}
\newtheorem*{teo*}{Theorem}

\newtheorem*{thmd}{Particular Case D}
\newtheorem*{thmb}{Particular Case B}
\newtheorem*{thmc}{Theorem C}
\newtheorem*{prop*}{Proposition}
\newtheorem{prop}[thm]{Proposition}

\theoremstyle{definition} 

\newtheorem*{dfn*}{Definition}
\newtheorem{dfn}{Definition}

\theoremstyle{remark}                   
 
\newtheorem{ntz}{Notation}[section]         

\title{Syzygies and equations of Kummer varieties}
\author{Sofia Tirabassi}
\address{Institute of Mathematics\\
Warsaw University\\
ul. Banacha 2 \\
 02-097 Warszawa
(Poland)}
\email{tirabassi@mimuw.edu.pl}
\keywords{Syzygies, Koszul cohomology, Kummer varieties.}

\subjclass[2010]{14M99, 14K99 Primary, 14K05, 14F05 Secondary}

\date{}

\definecolor{grigio}{gray}{.8}
\definecolor{viola}{rgb}{.4,0,.6}
\definecolor{rosso}{rgb}{.7,0,.1}
\definecolor{verde}{rgb}{0,.3,0}

\makenomenclature
\begin{document}
 \maketitle
\begin{abstract}
 \noindent We study property $N_p$ for Kummer varieties embedded by a power of an ample line bundle.
\end{abstract}
\section{Introduction}
Let $X$ be an abelian variety. Its \textit{associated Kummer
variety}, $K_X$, is the quotient of $X$ by the natural $(\Z/2\Z)$-action induced
by
the morphism $i_X:X\rightarrow X$ defined by $x\mapsto -x$.
Given a Kummer variety, $K_X$, and an ample line bundle $A$ on
$K_X$, a result of Sasaki (\cite{Sasaki1})
states that $A^{\otimes m}$ is very ample and the embedding it
defines is projectively normal as soon as $m\geq 2$. Later
Khaled (\cite{Khaled2}) and Kempf (\cite{Kempf1}) proved
that, under the same conditions, the homogeneous ideal of $K_X$ is
generated by elements of degree 2 and 3, while, if $m\geq 3$ it
is generated in degree 2. Concerning the case $m=1$, if, furthermore, we assume that $A$
is a normally generated, very ample, line bundle on $X$, then the
homogeneus ideal of $K_X$ is generated in degree less than or equal to
4. In this
paper we prove that these statements are particular cases of
more general results on the syzygies of the variety $K_X$.\par
More precisely, let $Z$ an algebraic variety over an
algebraically closed field
$k$ and let $\fas{A}$ an ample
invertible sheaf on $Z$, generated by its global sections. With
$R_\fas{A}$ we will indicate the \textit{section
ring
associated to
the sheaf
$\fas{A}$}:
$$R_\fas{A}:=\bigoplus_{n\in\mathbb{Z}}H^{0}({Z},{\fas{A}^{
\otimes n}}),$$
while $S_{\fas{A}}$ will be the symmetric algebra of
$H^{0}({Z},{\fas{A}})$. The ring $R_\A$ is a finitely
generated graded
$S_\fas{A}$-algebra and as such it admits a \textit{minimal free
resolution} 
\begin{equation}\label{minalfree}
E_\bullet = 0\rightarrow\cdots\xrightarrow{f_{p+1}}
E_p\xrightarrow{f_p}\cdots\xrightarrow{f_2} E_1\xrightarrow{f_1}
E_0\xrightarrow{f_0}R_\A\rightarrow 0
\end{equation}
with  $E_i\simeq\bigoplus_j S_\A(-a_{ij})$, $a_{ij}\in\mathbb  Z_{>0}$.
In order to extend classical results of Castelnuovo, Mattuck,
Fujita and Saint-Donat on the projective embeddings of curves,
Green (\cite{GI}) introduced the following:
\begin{dfn*}[Property $N_p$
]
Let $p$ be a given integer. The line bundle $\fas{A}$ satisfies
property $N_p$ if, in the notations above,
$$E_0=S_\fas{A}
$$
and
$$E_i=\oplus S_\fas{A}(-i-1)\quad 1\leq i\leq p.$$
\end{dfn*}
\noindent Pareschi (\cite{Pa}) extended the
above
condition as follows : we say that, given a non-negative integer $r$, property
$\sottosopra{N}{0}{r}$ holds for $\A$ if,
in
the notation above, $a_{0j}\leq 1+r$ for
every $j$. Inductively we say that $\A$ satisfies property
$N_p^r$
if $N_{p-1}^r$ holds for $\A$ and $a_{pj}\leq p+1+r$ for every
$j$.\\

Green in \cite{GI} proved that, if $Z$  is a
smooth curve of genus $g$ and $\A$ a very ample line bundle on
$Z$ then $\A$ satisfies $N_p$ if 
$\deg \A\geq 2g+1+p$. He also conjectured  the behaviour of
property $N_p$ for the canonical bundle of a non-hyperelliptic
curve. 
Green's conjecture was
recently proved for the general curve by Voisin (\cite{V1,V2}).\\
\indent For what concerns higher dimensional varieties, the syzygies
of the
projective space were studied by Green in \cite{GII}, 
\cite{OP2001}, and in the recent preprint \cite{Ein2011}. 
For
arbitrary smooth varieties there is a general conjecture of Mukai
and in \cite{EL} Ein--Lazarsfeld proved that, if $Z$ is of dimension $n$, denoting by
$\omega_Z$ its canonical line bundle on $Z$, then for any $\LL$ very
ample on $Z$ the sheaf
$$\A:=\omega_Z\otimes\LL^{\otimes (n+1+d)}$$
satisfies $N_p$ for every $d\geq p\geq 1$.\par
 Abelian varieties distinguish themselves among other smooth
varieties since, at least in what
concerns their syzygies, they tend to behave in any dimension
like elliptic curves.
More precisely, Koizumi (\cite{Ko1976}) proved
that, given $\A$ ample on an abelian variety $X$, then for $m\geq
3$, $\A^{\otimes m}$ embeds $X$ in the projective space as a
projectively normal variety. Furthermore, a classical theorem of
Mumford\index{Mumford D.} (\cite{Mum-ab}), improved by
Kempf (\cite{Ke1}), states that the homogeneous
ideal of $X$ is generated in degree 2 as long as $m\geq 4$. These
results inspired Lazarsfeld to conjecture
that $\A^{\otimes m}$ satisfies $N_p$ for every $m\geq p+3$. A generalized version of Lazarsfeld's
conjecture, involving property
$N_p^r$ rather than simply $N_p$, was proved in \cite{Pa}; later
in \cite{PPII} Pareschi--Popa
were able to recover and improve Pareschi's statements as a
consequence of the powerful,
Fourier-Mukai based,
theory of $M$-regularity that they
developed in \cite{PPI}.\par
Given the results on projective normality and degree of defining
equations of Sasaki and Khaled, it was natural to conjecture that a bound $m_0(p,r)$, \emph{independent of the
dimension of $K_X$},
could be found  such that $\A^{\otimes m}$ satisfies $N_p^r$
for every $m\geq m_0(p,r)$.
In this paper we present some results in this
direction. The main idea behind the proofs is that
ample line bundles on a Kummer variety $K_X$ have a nice
description in terms of ample line bundles on $X$. More
precisely, denoting by $\pi_X:X\rightarrow
K_X$ the quotient map, then for every $A$ ample on $K_X$ there
exists $\A$ ample on $X$ such that $\pi_X^*A\simeq\A^{\otimes
2}$. Hence we can use Pareschi--Popa machinery to find some
results on $\A^{\otimes 2m}$ and then study how the $\Z/2\Z$
action fits in the framework. Below we list the main
results we
obtained.

\begin{thma}
 Fix two non-negative integers $p$ and $r$ such that
$\mathrm{char}(k)$ does not divide $p+1,\; p+2$. Let $A$ be an
ample line
bundle on a Kummer variety $\kummerx$. Then:
\begin{enumerate}
\item[(a)] 
$A^{\otimes n}$ satisfies property
$N_p$ for every $n\in\Z$ such that $n\geq p+2$;
\item[(b)] more generally $A^{\otimes n}$ satisfies property
$N_p^r$ for every $n$ such that $(r+1)n\geq p+2$.
\end{enumerate}
\end{thma}

Since it consists in an improvement of existing results on
the degrees of defining equations of Kummer varieties it is
worth to emphasize individually the case $p=1$ of the above
statement. Thanks to the geometric meaning of property $N_p^r$
(Section \ref{npekos}) one can deduce the following:
\begin{thmb}
 Let be $A$ a very ample line bundle on a Kummer variety $K_X$. Then the ideal of the image $\varphi_A(K_X)$ in
$\mathbb{P}(\coo{0}{X}{A})$ is generated by forms of degree at
most 4.
\end{thmb}
The improvement with
respect to what was classically known, thanks to the work
of Wirtinger, Andreotti--Mayer \cite{AM1967} and Khaled, is that this holds for
all Kummer varieties, and not only for the general ones.
It might be interesting to see whether this fact has some applications to moduli of abelian varieties.\par
Adding one hypothesis about the line bundle $A$ we can get a
somewhat better result improving the
work of Kempf and Khaled;
namely:
\begin{thmc}
 Let  $p$ and $r$ be two integers such that $p\geq1$,  $r\geq
0$ and
$\mathrm{char}(k)$ does not divide $p+1,\; p+2$. Let $A$ be an
ample line
bundle on a Kummer variety $\kummerx$, such that its pullback
$\pi_X^*A\simeq \A^{\otimes 2}$ with $\A$ an ample symmetric
invertible sheaf on $X$ which does not have a base divisor. Then the followings hold:
\begin{enumerate}
\item[(a)] 
$A^{\otimes n}$ satisfies property
$N_p$ for every $n\in\Z$ such that $n\geq p+1$;
\item[(b)] more generally $A^{\otimes n}$ satisfies property
$N_p^r$ for every $n$ such that $(r+1)n\geq p+1$.
\end{enumerate}
\end{thmc}
Again, it is worth of single out  the case $p=1$ of the above
Theorem, concerning the
equations
of the Kummer variety $\kummerx$.
\begin{thmd}
 Suppose that $\mathrm{char}(k)$ does not divide 2 or 3 and let
$A$ be an ample invertible sheaf on $\kummerx$ such that
$\pi_X^*A\simeq\A^{\otimes 2}$ with $\A$ without a base divisor.
Then:
\begin{itemize}
 \item[(a)] if $n\geq2$, then the ideal $\fas{I}_{\kummerx,
A^{\otimes n}}$ of the embedding $\varphi_{A^{\otimes n}}$
is generated by quadrics;
\item[(b)] $\fas{I}_{\kummerx, A}$ is generated by quadrics and
cubics.
\end{itemize}

\end{thmd}

The key point of the proofs of Theorems 2.A and 2.C will be
to reduce the problem on the Kummer variety $K_X$ to a
different problem on the abelian variety $X$. Namely we will show
that property $N_p^r$ on the Kummer is
implied by the surjectivity of a map of the type:
\begin{equation}\label{eq:intro1}\tag{*}
\bigoplus_{[\alpha]\in\w
U}\coo{0}{X}{\F\otimes\alpha}\otimes\coo{0}{X}{\fas{H}\otimes\alpha}\stackrel{m_\alpha}{\rightarrow}\coo{0}{X}
{ \F\otimes\fas{H}\otimes\alpha}
\end{equation}
where $\F$ and $\fas{H}$ are sheaves on $X$ and $\w
U$ is a non empty open
subset of $\w X$, the abelian variety dual to $X$, and $m_\alpha$ is just the multiplication of global
sections. Criteria for the surjectivity of such maps are implicit
in Kempf's work\index{Kempf G.} (\cite{Ke5,Ke1990}), for the case
$\F$ a vector bundle and $\fas{H}$ a line bundle (for an
explicit argument due to Lazarsfeld see \cite{Pa}). These results
were improved and extended to general coherent sheaves by
Pareschi--Popa in
\cite{PPII}.\\

This paper is organized in the following manner: in the next
Section we explain some background material such as the relationship between proprerty $N_p^r$ and the
cohomology of the Koszul complex and a useful
criterion for the surjectivity of a map of type
\eqref{eq:intro1}. In Section 3 we present some slightly
modified version of results of Sasaki and Khaled. The last
section is entirely devoted to the proof of the
main theorems.
\subsubsection*{Notations} Throughout the paper we will work over an algebraically closed field $k$ of characteristic different from 2; further restrictions on the field will be stated when needed. The word "variety" will mean a projective variety over $k$.
\begin{itemize}
 \item[$\diamond$] Given $\F$ a sheaf on a variety $Z$, then we
denote by $\mathrm{Bs}(\F)$ the locus in $Z$ where $\F$ is not
generated by its global sections, i. e. the locus of $z\in Z$ where the evaluation map $H^0(Z,\F)\otimes k(z)\rightarrow \F\otimes k(z)$ fails to be surjective.
\item[$\diamond$] If $L$ is a line bundle on a variety
$Y$, $\varphi_L$ will be the rational map associated to $L$. The
symbol $\fas{I}_{Y,L}$ is the ideal of $\varphi_L(Y)$ in
$\mathbb{P}(H^0(Y,L))$.

\item[$\diamond$] If we have a product of varieties,
$X_1\times\cdots\times X_n$
then $p_i$ is the $i$-th projection. Given $\F_i$ a
sheaf on $X_i$, then $\F_1\boxtimes\cdots\boxtimes\F_n$
will stand for the sheaf $p_1^*\F_1\otimes\cdots\otimes
p_n^*\F_n$ on
$X_1\times\cdots\times X_n$.
\end{itemize}
Given an abelian variety $X$ of dimension $g$:
\begin{itemize}\item[$\diamond$] $i_X$ will be the inversion
morphism; $t_x$ will be the translation
by the point $x$;

\item[$\diamond$] The group law on the
abelian variety is denoted by $p_1+p_2$, while given an
integer $n$ the map $n_X:X\longrightarrow X$ will be the multiplication by $n$.
\end{itemize}

\section{Background Material}
\subsection[Property ${N_p^r}$ and Koszul
Cohomology]{Property $\mathbf{N_p^r}$ and Koszul
Cohomology}\label{npekos}

We begin by reviewing some well known relations between property
$N_p$, or more generally property $N_p^r$, 
and the surjectivity of certain multiplication maps of sections of vector bundles. Let $Z$ be a projective variety
and $L$ be an ample invertible sheaf on $Z$.
We know from \cite[Thm. 1.2]{GI} (see also \cite{Green} Thm. 1.2
or \cite[p. 511]{Laz}) that condition $N_p$
is equivalent to the exactness in the middle of the complex
\begin{equation}\label{eq:np1}
 \bigwedge^{p+1}H^0(L)\otimes H^0(L^{\otimes h})\rightarrow \bigwedge^{p}H^0(L)\otimes H^0(L^{\otimes
h+1})\rightarrow \bigwedge^{p-1}H^0(L)\otimes H^0(L^{\otimes h+2})
\end{equation}
for any $h\geq 1$. More  generally, condition $N_p^r$ is equivalent to exactness in the middle of
\eqref{eq:np1} for every $h\geq r+1$. Suppose that $L$ generated by
its global sections and consider the following exact sequence:
\begin{equation}
 \escorta{M_L}{H^0(L)\otimes\OO{Z}}{L}.\nomenclature[ch]{$M_L$}
{Kernel of the evaluation map of $L$}
\end{equation}

It is well known (see, for example  \cite{Laz} or \cite{Ein2011})that property $N_p^r$ is implied by the surjectivity of 
\begin{equation}\label{eq:quattro}
 \bigwedge^{p+1}H^0(L)\otimes H^0(L^{\otimes
h})\longrightarrow H^0\left(\bigwedge^pM_L\otimes L^{\otimes h+1}\right).
\end{equation}
It follows that if
\begin{equation}
 \coo{1}{Z}{\bigwedge^{p+1}M_L\otimes L^{\otimes h}}=0
\end{equation}
for any $h\geq r+1$, then condition $N_p^r$ is satisfied. If
$\mathrm{char}(k)$ does not divide $p$,
$\bigwedge^p\fas{E}$ is a direct summands of $\sottosopra{\fas{E}}{\:}{\otimes p}$ for any vector bundle
$\fas{E}$.  This yields

\begin{lem}\label{previous lemma}
 Assume that $\mathrm{char}(k)$ does not divide $p+1$.
\begin{itemize}\item[(a)] If $\coo{1}{Z}{\sottosopra{M}{L}{\otimes p+1}\otimes L^{\otimes h}}=0$ for any
$h\geq r+1$, then $L$ satisfies $N_p^r$.
 \item[(b)] Let $W\subseteq\coo{0}{Z}{L}$ be a free sublinear
system and denote by $M_W$ the kernel of
the evaluation map $W\otimes\OO{Z}\rightarrow L$. Assume that
$\coo{1}{Z}{\sottosopra{M}{W}{\otimes
p}\otimes L^{\otimes h}}=0$. Then
$\coo{1}{Z}{\sottosopra{M}{W}{\otimes p+1}\otimes L^{\otimes
h}}=0$ if and only if the multiplication map
\[
W\otimes H^0(\sottosopra{M}{W}{\otimes p}\otimes L^{\otimes
h})\longrightarrow
H^0(\sottosopra{M}{W}{\otimes p}\otimes L^{\otimes h+1})
\]
is surjective.
\end{itemize}

\end{lem}
\begin{proof}
 The proof of (a) is straightforward, while (b) follows from the following exact sequence:
\[
\escorta{\sottosopra{M}{W}{\otimes p+1}\otimes L^{\otimes
h}}{W\otimes\sottosopra{M}{W}{\otimes
p}\otimes L^{\otimes h}}{\sottosopra{M}{W}{\otimes p}\otimes
L^{\otimes h+1}}.
\]
\end{proof}

\subsubsection{Property $\mathbf{N_p^r}$ for
Small $\mathbf{p}$'s'}\label{rmk:geometric}
By definition , if a
variety $Z$ is embedded in a projective space by a very ample
line bundle $L$ satisfying property $N_0^r$,
then the variety $Z$ is $h$-normal for every $h\geq r$.  In particular
property $N_0$ is equivalent to projective normality.\par
Although this is a standard fact, it is perhaps worth spelling out the meaning of property $N_1^r$, thus providing a direct proof of Particular Case B and D.
\begin{prop}\label{prop:equations} If $L$ is a very ample line
bundle on an algebraic variety $Z$ satisfying
$N_1^r$, then the homogeneous ideal of $Z$ is generated by
homogeneous elements of degree at most $r+2$.
\end{prop}
\begin{proof}
Denote by $V$ the vector space $\coo{0}{Z}{L}$ and let $\mathrm{S}^k V$ be the component of
degree
$k$ of the symmetric algebra of $V$, $S$. Consider furthermore
the two $S$-modules
\[I=\bigoplus\coo{0}{\mathbb{P}(V)}{\fas{I}_{Z,L}(k)}\] and $R_L$. Look at the following commutative
diagram where the middle column is  given by the Koszul complex. \[
  \xymatrix{
& &&\mathrm{S}^{k-1}V\otimes\bigwedge^2V\ar[rr]^{(3)}\ar[d] &&H^0(Z,\;M_L\otimes L^{\otimes k})\ar[d] & \\
0\ar[r]& I_k\otimes V\ar[d]_{(1)}\ar[rr] &&\mathrm{S}^kV\otimes V\ar[d]_{(2)}\ar[rr] && \{R_L\}_k\otimes
V\ar[r]\ar[d] & 0\\
0\ar[r]& I_{k+1}\ar[rr]&&\mathrm{S}^{k+1}V\ar[rr] &&\{R_L\}_{k+1}\ar[r] & 0\\}
 \]
Our aim is to see that the map (1) is surjective for every $k\geq r+2$. Suppose that $L$ satisfies property
$N_1^r$, then in particular property $N_0^r$ holds for $L$ and
the second and third row are exact for every
$k\geq r+1$. Since (2) is surjective, by the Snake Lemma, for every $k\geq r+1$ the surjectivity of (1) is
implied by the surjectivity of (3) for every $k\geq r+2$. Now we can factor (3) in the following way:
\[\xymatrix{\mathrm{S}^{k-1}V\otimes\bigwedge^2V\ar[rr]^{(3)}\ar[dr]_g &&H^0(Z,\;M_L\otimes L^{\otimes
k})\\
&\coo{0}{Z}{L^{\otimes k-1}}\otimes\bigwedge^2V \ar[ur]_f&}
\]
where $g$ is the  canonical mapping $\coo{0}{\mathbb{P}(V)}{\fas{O}_{\mathbb{P}}(k-1)}\longrightarrow
\coo{0}{Z}{L^{\otimes k-1}}$ and $f$ is the map in \eqref{eq:quattro}.
For every $k\geq r+2$, $g$ is surjective  because $L$ satisfies
$N_0^r$, while the surjectivity of $f$
is equivalent to property $N_1^r$; hence (3) is surjective.
\end{proof}
\subsection[$M$-regular Sheaves and Multiplication
Maps]{$\mathbf M$-regular Sheaves and Multiplication Maps}
\subsection{A review on $M$-regularity}

In what follows we briefly recall some notions about
$M$-regular sheaves. We refer to
\cite{PPI,PPII} for further
details.\par
Given an abelian variety $X$ and a coherent sheaf $\F$ on $X$, we introduce the \emph{cohomological support loci} of $\F$: 
$$V^i(\F)
:=\{\alpha\:|\:h^i(\F\otimes\alpha)\neq 0\}\subseteq\w X$$
It is said that a sheaf $\F$ satisfies \emph{index theorem} with index $j$ (or
equivalently it is said to
be I.T.$(j)$ if for
every $i\neq j$ the loci $V^i(\F)$ are empty.
\begin{dfn}
 A sheaf $\fascio$ over an abelian variety $X$ is said to be \emph{Mukai-regular} (or simply
$M$-regular) if, for every $i>0$, $\mathrm{codim}\:V^i(\F)\geq i+1$.
\end{dfn}

\subsubsection{Application to multiplication maps}
 $M$-regular sheaves and $M$-regularity
theory, introduced by Pareschi-Popa in a long series of paper ranging from 2002 to 2009, are crucial to our
purpose thank
to their
application in determining whether a map of the form
\begin{equation}\label{eq:intro}
\bigoplus_{[\alpha]\in U}\multalpha{X}{\F}{\fas{H}},
\end{equation}
 with $\F$ and $\fas{H}$ sheaves on an abelian variety $X$ and
$U\subseteq \w X$ an open set, is surjective. We will list below
all the results of such kind
that we will be using troughout the paper. The first one is an
extension of a theorem that
had
already
appeared in the work of Kempf\index{Kempf G.},
Mumford\index{Mumford D.} and Lazarsfeld\index{Lazarsfeld R.}.
\begin{thm}[\cite{PPI}, Theorem 2.5]\label{lem:kempf-pa2.2}\label{cor:m-reg}
Let $\F$ and $\fas{H}$ be sheaves on $X$ such that $\F$ is
$M$-regular\index{M@$M$-regular} and $\fas{H}$ is locally free
satisfying
I.T. with index 0\index{I.T.}. Then 
\eqref{eq:intro}
is surjective for any non empty Zariski open set
$U\subseteq\widehat X$.\par
In particular, if $\F$ and $\fas{H}$ are as above, then there exists $N$ a positive integer such that for the
general $[\alpha_1],\ldots,[\alpha_N]\in\w X$ the map
\[
\xymatrix{
\displaystyle{\bigoplus_{k=1}^N}\cox{\F\otimes\alpha_k}
\otimes\cox{\fas{H}\otimes\alpha_k^\vee}\ar[rr]^{\hspace*{2cm}m_{
\alpha_k } } & &\cox{\F\otimes\fas{H}}
}
\]
is surjective.
\end{thm}
We conclude this paragraph by presenting two results on
multiplication maps of sections we will be needing afterward.
\begin{prop}\label{prop:blp2}
 Let $\A$ be an ample line bundle on an abelian variety $X$. The map
\begin{equation}\label{eq:blp2}
m:\cox{\A^{\otimes
2}}\otimes\cox{\A^{\otimes
2}\otimes\alpha}\longrightarrow\cox{\A^{\otimes
4}\otimes\alpha}.\end{equation}
is surjective for the general $[\alpha]\in\w X$. If furthermore
$\A$ does not have a base divisor, then the locus
$Z\subseteq\Pic^0(X)$ in which it fails to be surjective has
codimension at least 2.
\begin{proof}
 The first part of the statement is classical, for a reference
see, for example \cite[Proposition 7.2.2]{BL}. Pareschi--Popa
provided another proof of this, together with  a proof  of the
second statement of the Proposition, using their Fourier--Mukai
based methods (see \cite[Proposition 5.2, Proposition 5.6, and
Theorem 5.8]{PPII}).
\end{proof}

\end{prop}

\section{Multiplication Maps on Abelian
Varieties}\label{section:mult}
 Let $\fas{A}$ be
an ample symmetric invertible sheaf on the abelian variety $X$
and denote by $\psi_{\A}:\A\rightarrow i^*_X\A$ its \emph{normalized isomorphism}
(see \cite[p. 304]{Mum-eq}). The $\G$ action on $X$ induced by
the involution
$i_X:X\longrightarrow
X$\nomenclature{$\mathbf{-1}_X$}{Involution on $X$ given by $x\mapsto -x$}
induces trough $\psi_\A$ a lifting of the action on $\A$. The
composition
\[\xymatrix{H^0(X,\:\A)\ar[rr]^{i_X^*}&&H^0(X,\:i_X^*\A)\ar[rr]^{
i_X^*(\psi_\A) }&&H^0(X,\:\A)}\]  is
denoted by $[-1]_\A$. We let
\[H^0(X,\:\A)^\pm=\{s\in H^0(X,\:\A)\text{ such that
}[-1]_\A s=\pm s\}\nomenclature{$\coo{0}{X}{\fas{\A}}^\pm$}{the
eigenspace
associated to the eigenvalue $\pm 1$ of the involution $[-1]$}\]
If $\fas{A}$ is \emph{totally symmetric}
(see \cite[p. 304]{Mum-eq}, e.g. if $\A$ is an even power of a
symmetric line bundle), then there exists\footnote{This is a standard fact. A proof could be found in \cite[Prop. 2.1.7]{MyThesis}} a line bundle $A$ on
the
Kummer variety $\kummerx$ such that $\pi_X^*A\simeq\fas{A}$ and
one can identify $\coo{0}{\kummerx}{A}$ with
$\coo{0}{X}{\fas{A}}^+$. Conversely if $A$ is an ample line
bundle on $K_X$, then there exists $\A$, ample and symmetric on $X$,
such that $\pi_X^*A\simeq\A^{\otimes 2}$.\par
\indent A well known result by Khaled states that
\begin{prop}[\cite{Khaled2}]\label{khaledfr:1.4.4}
 Suppose that $k=2n$ is an even positive integer. Thus
$\fas{A}^{\otimes k}$ is totally symmetric. Then 
\[
m_\alpha^+:H^0(X,\fas{A}^{\otimes
k})^+\otimes\coo{0}{X}{\fas{A}^{\otimes
h}\otimes\alpha}\longrightarrow\coo{0}{X}{ \fas{A}^{\otimes k+h
}
\otimes\alpha }
\]
is surjective for 
$n\geq 1$, every $h\geq 3$, and for every $\alpha\in\Pic^0(X)$
\end{prop}
The main goal of this section is to prove that the same is true
for every $h= 2 $ and for \emph{general}
$\alpha\in\w X$. If furthermore we assume that $\A$ does not
have a base divisor, then we will show that the locus of
$[\alpha] \in \w X$ where
\begin{equation}\label{eq:moltiplicazione}
 m_\alpha^+:H^0(X,\fas{A}^{\otimes
2n})^+\otimes\coo{0}{X}{\fas{A}^{\otimes
2}\otimes\alpha}\longrightarrow\coo{0}{X}{ \fas{A}^{\otimes
2(n+1) }
\otimes\alpha }
\end{equation}
fails to be surjective has codimension at least 2. We will do
this by the methods of Khaled (\cite{Khaled2}).\par
To this end, consider the isogeny
\[\xi:X\times X\longrightarrow X\times X\]
\nomenclature[cnaa]{$\xi$}{Isogeny $X\times X\rightarrow X\times
X$
given by $(p_1+p_2,\;p_1-p_2)$}
given by $\xi=(p_1+p_2,\;p_1-p_2)$
\begin{lem}
 For any $[\alpha]\in\w X$ we have an isomorphism
\[\xymatrix{
\xi^*(p_1^*(\A\otimes\beta)\otimes
p_2^*(\A\otimes\alpha))\ar[rr]
&&p_1^*(\A^{\otimes2}\otimes\beta\otimes\alpha)\otimes
p_2^*(\A^{\otimes 2}\otimes\beta\otimes\alpha^\vee)}\]
\end{lem}
\begin{proof}
 Observe that
\begin{align*}\xi^*(p_1^*(\A\otimes\beta)\otimes
p_2^*(\A\otimes\alpha))_{|X\times\{y\}}&\simeq t_y^*\A\otimes
t_{-y}^*\A\otimes t_y^*\alpha^\vee\otimes
t_{-y}^*\alpha\simeq\\
&\simeq\A^{\otimes 2}\otimes
\beta\otimes \alpha.
\end{align*}
Restricting to $\{0\}\times X$ we get
\begin{align*}\xi^*(p_1^*(\A\otimes\beta)\otimes
p_2^*(\A\otimes\alpha))_{|\{0\}\times X}&\simeq
\A\otimes\beta\otimes i^*
\A\otimes \alpha)\simeq \A^{\otimes 2}\otimes\beta\otimes
\alpha^\vee;
\end{align*}
the statement follows from the See-Saw principle.
\end{proof}

Composing with the K\"unneth isomorphism we have a map
\[\xi^*:\cox{\A\otimes\beta}\otimes\cox{\A\otimes\alpha}
\longrightarrow
\cox{\A^{\otimes2}\otimes\beta\otimes\alpha}\otimes\cox{\A^{
\otimes
2}\otimes\beta\otimes\alpha^\vee}.\]

Taking $[\alpha]=[\beta^\vee]$  we want to characterize the
image of $\cox{\A\otimes\beta}\otimes\cox{\A\otimes\beta^\vee}$
in  $\cox{\A^{\otimes2}}^+
\otimes\cox{\A^{\otimes2}\otimes\beta^{\otimes2}}$ through
$\xi^*$.

In order to achieve this goal, we consider $\w T^{\A\otimes\beta}$, the involution of $\cox{\A\otimes\beta}\!\otimes\cox{\A\otimes\beta^\vee}$ defined by $\w
T^{\A\otimes\beta}(s\otimes t)=i^*(\psi_{\beta^\vee})i^*t\otimes
i^*(\psi_{\beta})i^*s$.
 Let us denote by
$$[\cox{\A\otimes\beta}\otimes\cox{\A\otimes\beta^\vee}]^{\pm}$$
the eigenspaces.
\begin{prop}\label{xi equivariant}
 For every $[\alpha]\in\w X$ we have
\[\xi^*[\cox{\A\otimes\beta}\otimes\cox{\A\otimes\beta^\vee}]
^\pm%
\subseteq\cox{\A^{\otimes 2}}^\pm\otimes\cox{ \A^{\otimes
2}\otimes\beta^{\otimes2}}.  \]

\end{prop}
The proof of this statement is just a slight modification of Khaled's proof of
\cite[Proposition 2.2]{kh} and therefore we omit it. The complete proof can also be found in \cite[Prop. 2.2.3]{MyThesis}

\begin{thm}\label{thm:bello}
 Let $\A$ be an ample symmetric line bundle on $X$. Take
$[\alpha]\in\w X$.
Then the multiplication map
\[m:\cox{\A^{\otimes 2}}\otimes\cox{\A^{\otimes
2}\otimes\alpha}\longrightarrow\cox{\A^{\otimes
4}\otimes\alpha}\]
is surjective if and only if the following multiplication map is
surjective
\[m^+:\cox{\A^{\otimes 2}}^+\otimes\cox{\A^{\otimes
2}\otimes\alpha}\longrightarrow\cox{\A^{\otimes
4}\otimes\alpha}\]
\end{thm}
\begin{proof} Obviously, if the map $m^+$ is surjective, then also $m$ is. Therefore we are left to prove that the surjectivity of $m$ implies the surjectivity of $m^+$.
As observed by Khaled, $m$ is the composition of
$$\xi^*:H^0(X,\A^{\otimes 2})\otimes H^0(X,\A^{\otimes 2}\otimes
\alpha)\rightarrow H^0(X,\A^{\otimes 4}\otimes \alpha)\otimes
H^0(X,\A^{\otimes 4}\otimes \alpha^\vee)$$ with 
$\mathrm{id}\otimes\mathrm{e}_{\A^{\otimes
4}\otimes\alpha^\vee}$, where $\mathrm{e}_{\A^{\otimes
4}\otimes\alpha^\vee}$ denotes the evaluation in 0 of
the sections  of $\A^{\otimes
4}\otimes\alpha^\vee$. In fact, if we denote by
$\Delta:X\rightarrow X\times X$ the diagonal immersion, then
$m$ is just $\Delta^*$ composed with the K\"unneth isomorphism.
Now we can write $\Delta=\xi\circ f$ where $f:X\rightarrow
X\times X$ is the morphism defined by $x\mapsto (x,0_X)$. Now
observe that, modulo
K\"unneth isomorhism,
 $$f^*:H^0(X,\A^{\otimes 4})\otimes 
H^0(X, \A^{\otimes 4}\otimes\alpha^{\vee})\rightarrow
H^0(X,\A^{\otimes 4})$$ 
 is exactly
$\mathrm{id}\otimes\mathrm{e}_{\A^{\otimes
4}\otimes\alpha^\vee}$.\\
\indent Observe that $\xi\circ\xi$ is the map $(x,y)\mapsto
(2_X(x),2_X(y))$. Hence, taking $\beta\in\Pic^0(X)$ such that
$\beta^{\otimes 2}\simeq \alpha$, we can write 
$$m\circ\xi^*:H^0(X,\A\otimes\beta)\otimes
H^0(X,\A\otimes\beta^\vee)\rightarrow H^0(X;\A^{\otimes
4}\otimes \alpha)$$
as $\mathrm{id}\otimes\mathrm{e}_{\A^{\otimes
4}\otimes\alpha^\vee}\circ (2_X^*(-)\otimes
2_X^*(-))=2_X^*\otimes\mathrm{e}_{\A^{\otimes
2}\otimes\beta^\vee}$. Thus we can consider the following
commutative diagram 
\[
\xymatrix{%
\displaystyle{\bigoplus_{[\beta^{\otimes 2}]= [\alpha]}}
H^0(\A\otimes\beta)\otimes
H^0(\A\otimes\beta^\vee)\ar[rr]^{\xi^*}\ar[dr]_{
2_X^*\otimes\mathrm{e}_{\A\otimes\beta^\vee} } & &
H^0(\A^{\otimes 2})\otimes
H^0(\A^{\otimes 2}\otimes\alpha)\ar[dl]^{m}\\
 &H^0(\A^{\otimes 4}\otimes\alpha) &
}
\]
where the upper arrow is an isomorphism by projection formula. It follows that the surjectivity of $m$ is equivalent to the surjectivity of 
$$
\bigoplus_{[\beta^{\otimes
2}]=[\alpha ]}2_X^*\otimes\mathrm{e}_{\A\otimes\beta^\vee}
:\bigoplus_{[\beta^{\otimes 2}]= [\alpha]}
H^0(\A\otimes\beta)\otimes
H^0(\A\otimes\beta^\vee)\rightarrow H^0(\A^{\otimes 4}\otimes\alpha),
$$
which, in turn, is equivalent the following condition:
\begin{equation}\tag{\dag}\label{condizione}
 \text{$0\notin\mathrm{Bs}(\A\otimes\beta^\vee)$ for every
$\beta\in\w X$ such that $\beta^{\otimes 2}\simeq\alpha$}.
\end{equation}
We claim that \eqref{condizione} implies the surectivity of  $m^+$, the statement of the theorem will follow directly.

\indent Thanks to Proposition \ref{xi equivariant} and
the
isomorphism
\[\xymatrix{
\bigoplus_{[\beta^{\otimes 2}]\simeq[\alpha]}
H^0(\A\otimes\beta)\ar[rr]^{2_X^*} &&
H^0(\A^{\otimes4}\otimes\alpha),}\]
yielded by $2_X^*$ and projection formula, it is enough to check
that for every $[\beta]$ satisfying
$[\beta^{\otimes2}]=[\alpha]$
\[2_X^*\otimes\mathrm{e}_{\A\otimes\beta^\vee}
[H^0(\A\otimes\beta)\otimes
H^0(\A\otimes\beta^\vee)]^+=2_X^*(H^0(\A\otimes\beta)). 
\]
Hence for each $s\in H^0(\A\otimes\beta)$, we need a section $\sigma \in (H^0(\A\otimes\beta)\otimes
H^0(\A\otimes\beta^\vee))^+$ such that
$2_X^*\otimes\mathrm{e}_{\A\otimes\beta^\vee}
(\sigma)=2_X^*(s)$.
Therefore take $s\in H^0(X,\A\otimes\beta)$ and denote by
$\lambda$ the constant
$\mathrm{e}_{\A\otimes\beta^\vee}([i^*\psi_\beta\circ
i^*](s))$. If $\lambda\neq 0$, 
take $$\sigma:=\frac{1}{\lambda}\cdot(s\otimes[i^*\psi_\beta\circ
i^*](s)).$$
 If, otherwise, 
$\lambda= 0$, it follows from \eqref{condizione}
that there exists $t\in H^0(X, \A\otimes\beta)$ such that
$\mathrm{e}_{\A\otimes\beta^\vee}([i^*\psi_\beta\circ
i^*](t))=1$. Then, take $\sigma$ to be the section
\[ (s+t) \otimes [i^*\psi_\beta\circ
i^*]\left(s+t\right)-t\otimes[i^*\psi_\beta\circ i^*](t)\quad
\in\quad (H^0(\A\otimes\beta)\otimes
H^0(\A\otimes\beta^\vee))^+.\]
Applying $2_X^*\otimes\mathrm{e}_{\A\otimes\beta^\vee
} $ to $\sigma$ we get
\begin{align*}
2_X^*\otimes\mathrm{e}_{\A\otimes\beta^\vee
}(\sigma)&= 2_X^*\otimes\mathrm{e}_{\A\otimes\beta^\vee
} \left( (s+t) \otimes [i^*\psi_\beta\circ
i^*]\left(s+t\right)-t\otimes[i^*\psi_\beta\circ i^*](t)
\right)=\\&=2_X^*(s+t)\cdot 1-2_X^*(t)\cdot 1=2_X^*(s).
\end{align*}
\end{proof}
\begin{cor}\label{cor:bello}
\begin{enumerate}
 \item For every $\A$ ample symmetric invertible sheaf on $X$
the multiplication map
\begin{equation}\label{eq:mpiu}m^+:\cox{\A^{\otimes
2}}^+\otimes\cox{\A^{\otimes
2}\otimes\alpha}\longrightarrow\cox{\A^{\otimes
4}\otimes\alpha}\end{equation}
is surjective for the generic $[\alpha]\in\Pic^0(X)$.
\item If furthermore $\A$ does not have a base divisor, then the
locus where \eqref{eq:mpiu} is not surjective has codimension at
least 2.
\end{enumerate}
\end{cor}
\begin{proof}
 The statement follows directly from the Theorem \ref{thm:bello}
and the
corresponding statements (see Proposition \ref{prop:blp2}) about
the map $m$.
\end{proof}
Now we are ready to prove the result that is the main point
of this section.
\begin{thm}\label{mykhaled: 2}\label{prop:bella}
Let $\A$ be an ample symmetric invertible sheaf on $X$, then
\begin{enumerate}
 \item
there exist a non-empty open subset $U\subseteq\Pic^0(X)$ such that for every $n\in \Z$
with $n\geq 1$ and 
and every $\alpha\in U$ the following map is surjective
\begin{equation}\label{eq: multalphapiu}
m_\alpha^+:H^0(X,\fas{A}^{\otimes
2n})^+\otimes\coo{0}{X}{\A^{\otimes
2}\otimes\alpha}\longrightarrow\coo{0}{X}{\A^{\otimes 2n+2 }
\otimes\alpha };
\end{equation}
\item if, furthermore, $\A$ does not have a base divisor, then the
locus $Z$ in $\Pic^0(X)$ where \eqref{eq: multalphapiu} fails to
be surjective has codimension at least 2.
\end{enumerate}

\end{thm}

\begin{proof}
We will proceede by induction on
$n$, with base case given by Corollary \ref{cor:bello}.\\

\textit{Case $n>1$}. Observe the following commutative diagram:
\begin{footnotesize}
\[
\xymatrix{\coo{0}{X}{\A^{\otimes 2}}^+\hspace{-1.5mm}\otimes\coo{0}{X}{\A^{\otimes 2(
n-1)}}^+\hspace{-1.5mm}\otimes\coo{0}{X}{\A^{\otimes
2}\otimes\alpha}\ar[r]^{\hspace{11mm}\varphi_\alpha}\ar[d] &
\coo{0}{X}{{\A}^{\otimes
2}}^+\hspace{-1.5mm}\otimes\coo{0}{X}{\A^{\otimes
2(n-1)+2}\otimes\alpha}\ar[d]^{
\psi_\alpha } \\
\coo{0}{X}{\A^{\otimes
2n}}^+\hspace{-1.5mm}\otimes\coo{0}{X}{\A^{\otimes
2}\otimes\alpha}\ar[r]_{\hspace{8mm}m_{
\alpha } }
&\coo{0}{X}{\A^{\otimes 2n+2}\otimes \alpha}
}
\]
\end{footnotesize}
Let $[\alpha]\in\Pic^0(X)$ be is such that $m_\alpha$ is not surjective. Then, necessarily the map $\psi_\alpha\circ\varphi_\alpha$ is not surjective. Since $2+2(n-1)\geq 3$, by Proposition \ref{khaledfr:1.4.4},
$\psi_\alpha$ is surjective for every $[\alpha]\in\w{X}$. Per force, then, the map $\varphi_\alpha$ could not be surjective. It follows that the locus of point $[\alpha]$ such that $m_\alpha$ is not surjective is contained in the following set:
$$W:=\{[\alpha]\in\Pic^0(X)\:
|\:\varphi_\alpha\;\text{is not
surjective}\}.$$ 
By inductive hypothesis $W$ has codimension $\ge 1$, in the general case, or $\ge 2$, when
$\A$ has not a base divisor. Therefore the Theorem is proved.
\end{proof}

\section{Equations and Syzygies of Kummer
Varieties}\label{section:mreg}
Putting together the results of the previous paragraphs, in this
last Section we will prove Theorem A and Theorem C. First of all, observe that the
case $p=0$ of the Theorem A follows
directly as a corollary of Khaled's work (see Proposition
\ref{khaledfr:1.4.4}).
Thus we may suppose $p\geq1$.\par
Our strategy will be to use the part (b)
of Lemma \ref{previous lemma} and to reduce the problem to checking
the surjectivity of
\begin{equation}\label{eq:mult kummer}
\multshortk{A^{\otimes n}}{\sottosopra{M}{A^{\otimes n}}{\otimes p}\otimes A^{\otimes
nh}}{\sottosopra{M}{A^{\otimes n}}{\otimes
p}\otimes A^{\otimes n(h+1)}}
\end{equation}

\noindent for every $h\geq r+1$. Let $\A$ be an ample symmetric line
bundle on $X$ such that $\A^{\otimes
2}\simeq \pi^*A$, we split the proof in two steps. \\
\indent In first place, the surjectivity of \eqref{eq:mult kummer} is implied by the surjectivity of
\begin{small}
\begin{equation}\label{eq:*}
\multshortp{\A^{\otimes2n}}{\pi^*(M_{A^{\otimes n}})^{\otimes
p}\otimes\A^{\otimes2nh}}{\pi^*(M_{A^{\otimes n}})^{\otimes
p}\otimes\A^{\otimes 2n(h+1)}}.
\end{equation}
\end{small}
Before going any further we need to introduce some notation.
\begin{ntz}
Suppose that $A$ is an ample line bundle on $\kummerx$ and let
$\A$ be an invertible sheaf on $X$ such that
$\pi_X^*A=\A^{\otimes2}$. Take $n$ an integer such that
$A^{\otimes n}$ is globally generated and consider the 
following exact sequence of vector bundles:
$$
\escorta{M_{A^{\otimes n}}}{\coo{0}{\kummerx}{A^{\otimes
n}}\otimes\OO{\kummerx}}{A^{\otimes n}}.
$$
By pulling back via the canonical surjection $\pi_X$ we get: 
\[\escorta{\pi_X^*(M_{A^{\otimes
n}})}{\coo{0}{X}{\A^{\otimes
2n}}^+\otimes\OO{X}}{\A^{\otimes 2n}}.\]
Hence, after defining,
\[W_n:=\cox{\A^{2n}}^+\]\nomenclature[cnb]{$W_n$}{$\cox{\A^{2n}}
^+$ }
we have that $\pi_X^*M_{A^{\otimes n}}\simeq
M_{W_n}$\nomenclature[cnc]{$M_{W_n}$}{$\pi_X^*M_{A^{\otimes
n}}$}.
\end{ntz}
From now on given a sheaf $\F$ on $X$, we will often write
$H^0(\fas{F})$
instead of $\cox{\F}$.\par 
We will split the proof in two steps.

\noindent\textbf{Step 1:} \textit{Reduction to 
an $M$-regularity problem}

Consider the vanishing locus $V^1(M_{W_n}\otimes\A^{\otimes2})$.
We claim that it coincides with
the locus of $[\alpha]\in\w X$ such that
the multiplication map
\[
m_\alpha^+:H^0(\fas{A}^{\otimes2n})^+\otimes\Coo{\A^{\otimes2}
\otimes\alpha}\longrightarrow\Coo{\A^{
\otimes 2(n+1) }
\otimes\alpha }
\]
is not surjective. In fact, from the short exact sequence
\[\escorta{M_{W_n}\otimes\A^{\otimes
2}\otimes\alpha}{W_n\otimes\A^{\otimes
2}\otimes\alpha}{\A^{\otimes 2(n+1)}\otimes\alpha}.\]
taking cohomology we deduce that the surjectivity fo $m_\alpha^*$ is equivalent to the vanishing of $H^1(M_{W_n}\otimes\A^{\otimes
2}\otimes\alpha)$. 
Thanks to this characterization of the locus on $\Pic^0(X)$
where $m_\alpha^+$ fails to be surjective we were able to prove the following
\begin{lem}\label{lem:redmreg}
 Let $\A$ and $\fas{E}$ be an ample symmetric sheaf on an abelian
variety $X$ and
  a coherent sheaf on $X$, respectively. If
$\fas{E}\otimes\A^{\otimes -2}$ is $M$-regular, then  the
multiplication map
\begin{equation}\label{eq:**}H^0(X,\:\A^{\otimes 2n})^+\otimes
H^0(X,\fas{E})\longrightarrow H^0(X,\:\A^{\otimes
2n}\otimes\fas{E})\end{equation}
is surjective for every $n\geq 1$.
\end{lem}
Before proceeding with the proof we will state an immediate
corollary of this Lemma, that reduces our problem to a
$M$-regularity problem.
\begin{cor}\label{lem:reduction}
 If $M_{W_n}^{\otimes
p}\otimes\A^{\otimes
2(nh-1)}$ is $M$-regular, then \eqref{eq:*} is surjective.
\end{cor}
\begin{proof}[Proof of the Lemma.]
By Proposition \ref{mykhaled: 2}(1) we know that
$V^1(M_{W_n}\otimes\A^{\otimes2})$ a proper closed subset of $\w
X$. Therefore there exists an open set $\w{U}_0\subseteq\Pic^0(X)$ such that $\w{U}_0\cap
V^1(M_{W_n}\otimes\A^{\otimes2})=\emptyset$.
Now observe the following commutative diagram.
\[
\xymatrix{%
\displaystyle{\bigoplus_{\alpha\in\widehat{U}_0}}\Coo{
\A^{\otimes 2n}}^+\otimes\Coo{
\A^{\otimes2}\otimes
\alpha}\otimes\Coo{\fas{E}\otimes\A^{\otimes-2}\otimes\alpha^\vee
} \ar[dr] %
\ar[dd]_f
&\\
&\hspace{-10mm}\Coo{\A^{\otimes2n}}^+\otimes\Coo{
\fas{E}}\ar[dd]^g \\
\displaystyle{\bigoplus_{\alpha\in\widehat{U}_0}}\Coo{\A^{
\otimes 2n+2}\otimes
\alpha}\otimes\Coo{\fas{E}\otimes\A^{\otimes-2}
\otimes\alpha^\vee } \ar [dr]^h&\\
&\hspace{-0mm}\Coo{\fas{E}\otimes\A^{\otimes2n}}
}
\]

\noindent The map $f=\oplus \sottosopra{m}{\alpha}{+}$ is surjective by our choice of the set
$\w{U}_0$, the map $h$ is surjective by
$M$-regularity hypothesis together with Theorem
\ref{lem:kempf-pa2.2}. Thus $g$ is necessarily surjective.

\end{proof}

\hfill\\
\noindent\textbf{Step 3:} \textit{Conclusions.}
By Corollary \ref{lem:reduction} we are reduced to prove that $M_{W_n}^{\otimes p}\otimes\A^{\otimes 2(nh-1)}$ is
$M$-regular for every $h\geq r$. This is proved in the next two statements.
\begin{prop}\label{prop: it0}
 Let $p$ be a positive integer. Then 
$M_{W_n}^{\otimes
p}\otimes\A^{\otimes m}$ satisfies I.T. with index 0 (and
hence it is
$M$-regular) for every $m\geq
2p+1$
\end{prop}
Note that Theorem A follows at once from
this Proposition taking $m=2nr-2$.
\begin{proof}
 We will proceede by induction on $p$.\par
\textit{Case $p=1$}. Let us consider the  following exact sequence:
\begin{equation}\label{eq:star}
 \escorta{M_{W_{n}}\otimes\A^{\otimes
m}\otimes\alpha}{W_n\otimes\A^{\otimes m}
\otimes\alpha } { \A^{\otimes 2n+m } \otimes\alpha}.
\end{equation}
One can easily see that the vanishing of the higher
cohomology of $M_{W_{n}}\otimes\A^{\otimes
m}\otimes\alpha$ is implied by:
\begin{itemize}\item[(i)] the vanishing of the higher cohomology
of
$\A^{\otimes m}\otimes\alpha$ and
\item[(ii)] the surjectivity of the following multiplication map:
\[\multshortp{\A^{\otimes 2n}}{\A^{\otimes m}\otimes
\alpha}{\A^{\otimes 2n+m}\otimes\alpha}.\]
\end{itemize}
Condition (i) holds for every $\alpha$ as long as $m\geq1$,
while, thanks to Khaled result (see Proposition
\ref{khaledfr:1.4.4}), we know condition (ii) is satisfied for every
$\alpha$ as
long as $m\geq3$.\par
\textit{Case $p>1$}. Suppose now that $p>1$ and take any $\alpha\in\widehat X$. By twisting
\eqref{eq:star}
by $M_{W_n}^{\otimes p-1}$ we can observe that the vanishing of
higher cohomology of
$M_{W_{n}}^{\otimes p}\otimes\A^{\otimes
m}\otimes\alpha$  is implied by
\begin{itemize}\item[(i)] the vanishing of the higher cohomology
of
$M_{W_{n}}^{\otimes p-1}\otimes\A^{\otimes
m}\otimes\alpha$ and
\item[(ii)] the surjectivity of the following multiplication map:
\end{itemize}
\vspace{-3mm}
\begin{equation}\label{eq:ultima}\multshortp{\A^{\otimes
2n}}{M_{W_{n}}^{\otimes p-1}\otimes\A^{\otimes
m}\otimes\alpha}{M_{W_{n}}^{\otimes p-1}\otimes\A^{\otimes
2n+m}\otimes\alpha}.
\end{equation}

\noindent
By induction (i) holds as long as $m\geq 2p-1$. Thanks to Lemma
\ref{lem:reduction}
and
Lemma \ref{lem:kempf-pa2.2} we know that  if $M_{W_{n}}^{\otimes
p-1}\otimes\A^{\otimes
m-2}\otimes\alpha$ satisfies I.T. with index 0, then 
\eqref{eq:ultima}
is surjective. But we use induction again and we get that
$M_{W_{n}}^{\otimes
p-1}\otimes\A^{\otimes
m-2}\otimes\alpha$ is I.T. with index 0 whenever $m-2\geq2p-1$,
that is
whenever $m\geq2p+1$ and hence the statement is proved.
\end{proof}
\begin{prop}
 In the notations above, take $p\geq 1$ an integer. If
 $\A$ does not have a base divisor, then
$M_{W_n}^{\otimes p}\otimes \A^{\otimes m}$ is $M$-regular for
every $m\geq 2p$.
\end{prop}
Again Theorem C, follows at once after taking $m=2(nr-1)$.
\begin{proof}
For $m\geq 2p+1$ the statement is a direct consequence of the
Proposition above, hence we can limit ourselves to the
case $m=2p$. We will proceede by induction on $p$. \\

\textit{Case $p=1$}: We want to prove that $\mathrm{codim}\:
V^i(M_{W_n}\otimes\A^{\otimes 2})>i$ for every $i\geq1$. From
the vanishing of the higher cohomology of
$\A^{\otimes 2}\otimes\alpha$ for every $\alpha\in\Pic^0(X)$ we
know that the loci 
\[V^i(M_{W_n}\otimes\A^{\otimes 2})=\emptyset\quad\text{for
every $i\geq 2$}.\]
Recall that the locus $V^1(M_{W_n}\otimes\A^{\otimes 2})$ is the
locus of
points $\alpha\in\w X$ such that the multiplication
\[W_n\otimes H^0(\A^{\otimes 2}\otimes \alpha)\longrightarrow
H^0(A^{\otimes 2n+2}\otimes\alpha)\]
is not surjective. We know from Proposition
\ref{prop:bella} that, if $\A$ has not a base divisor,
then this locus has at least codimension 2 and hence the
statement is proved.\\

\textit{Case $p>1$} Take
$\alpha\in\w X$. Consider the following exact sequence
\begin{small}
\begin{equation}\label{eq:ind}\escorta{M_{W_n}^{\otimes
p}\otimes\A^{\otimes
2p}\otimes\alpha}{W_n\otimes M_{W_n}^{\otimes
p-1}\otimes\A^{\otimes
2p}\otimes\alpha}{M_{W_n}^{\otimes p-1}\otimes\A^{\otimes
2p+n}\otimes\alpha}\end{equation}
\end{small}
From Proposition \ref{prop: it0}(a) we know that for every
$i\geq 1$ both $H^i(M_{W_n}^{\otimes p-1}\otimes\A^{\otimes
2p}\otimes\alpha)$ and $H^i(M_{W_n}^{\otimes
p-1}\otimes\A^{\otimes
2p+n}\otimes\alpha)$ vanish. Thus the loci $V^i(M_{W_n}^{\otimes
p}\otimes\A^{\otimes
2p})$ are empty for every $i\geq 2$. It remains to
show that that 
$$\codim\: V^1(M_{W_n}^{\otimes
p}\otimes\A^{\otimes
2p})\geq 2.$$ 
As before one may observe that this
locus is exactly the locus in $\w X$ where the following
multiplication map fails to be surjective:
\[W_n\otimes H^0(M_{W_n}^{\otimes p-1}\otimes\A^{\otimes
2p}\otimes\alpha)\longrightarrow H^0(M_{W_n}^{\otimes
p-1}\otimes\A^{\otimes
2p+2n}\otimes\alpha).\]
In fact, taking cohomology in \eqref{eq:ind} and observing that
for every $[\alpha]\in\w X$, $h^1( M_{W_n}^{\otimes
p-1}\otimes\A^{\otimes
2p}\otimes\alpha)=0$, due to  Proposition \ref{prop: it0}, we
have the conclusion following the same argument as in the case
$p=0$.\par
Now take $[\alpha]\in V^1(M_{W_n}^{\otimes
p}\otimes\A^{\otimes
2p})$. By inductive hypothesis the sheaf $M_{W_n}^{\otimes
p-1}\otimes\A^{\otimes
2(p-1)}$ is $M$-regular.
Corollary \ref{cor:m-reg} implies that there exist a
positive
integer $N$ and $[\beta_1],\ldots,[\beta_N]\in\w X$ such that the
following is surjective.
\begin{small}
\[
\xymatrix{
\displaystyle{\bigoplus_{k=1}^N}H^0({
\A^{\otimes 2n+2}\otimes\beta_k\otimes\alpha})
\otimes H^0({M_{W_n}^{\otimes
p-1}\otimes\A^{\otimes
2(p-1)}\otimes\beta_k^\vee})\ar[r]^{\hspace*{2.5cm}m_{
\beta_k } }  &H^0({M_{W_n}^{\otimes
p-1}\otimes\A^{\otimes
2p+2n}\otimes\alpha})
}
\]
\end{small}
Consider the commutative square
\begin{footnotesize}
\[
\xymatrix{%
\displaystyle{\bigoplus_{k=1}^N}\Coo{
\A^{\otimes 2n}}^+\otimes\Coo{
\A^{\otimes2}\otimes
\alpha\otimes\beta_k}\otimes\Coo{M_{W_n}^{\otimes
p-1}\otimes\A^{\otimes2p-2}\otimes\beta_k^\vee}\ar[dr]%
\ar[dd]
&\\
&\hspace{-35mm}\Coo{\A^{\otimes2n}}^+\otimes\Coo{
M_{W_n}^{\otimes
p-1}\otimes\A^{\otimes2p}\otimes\alpha}\ar[dd] \\
\displaystyle{\bigoplus_{k=1}^N}\Coo{\A^{
\otimes 2n+2}\otimes
\alpha\otimes\beta_k}\otimes\Coo{M_{W_n}^{\otimes
p-1}\otimes\A^{\otimes2(p-1)}\otimes\beta_k^\vee}\ar[dr]&\\
&\hspace{-12mm}\Coo{M_{W_n}^{\otimes
p-1}\otimes\A^{\otimes2n +2p}}
}
\]\end{footnotesize}
The right arrow is not surjective by our choice of $\alpha$. The
bottom arrow is surjective, hence the left arrow could not be
surjective. Therefore $$\alpha\in
\bigcup_{k=1}^NZ_k,$$ where $Z_k$ stands for the locus of 
$[\beta]\in\w X$ such that the multiplication map
\begin{equation}\label{eq:rmk}\Coo{
\A^{\otimes 2n}}^+\otimes\Coo{
\A^{\otimes2}\otimes
\beta\otimes\beta_k}\rightarrow\Coo{\A^{
\otimes 2n+2}\otimes
\beta\otimes\beta_k}\end{equation}
   
fails to be surjective. Thus one has
that
\[V^1(M_{W_n}^{\otimes
p}\otimes\A^{\otimes
2p})\subseteq\bigcup_{k=1}^NZ_k.
\]
By Theorem \ref{prop:bella}(2) the loci $Z_k$ have
codimension at least 2, therefore
$$\codim\:V^1(M_{W_n}^{\otimes
p}\otimes\A^{\otimes
2p})\geq\codim\;\bigcup_{k=1}^NZ_k\geq 2$$ and
the
statement is proved.
\end{proof}

\subsection*{Acknowledgments}
This is part of my Ph. D thesis defended at the Mathmatics department of Università degli studi Roma TRE. I would like to thank my advisor G. Pareschi for introducing me to the subject of syzygies and for many suggestions. I am very indebted to C. Ciliberto and F. Viviani for many engaging mathematical conversations.\par
Part of this work was carried out during the tenure of an ERCIM "Alain Bensoussan" Fellowship Programme. This Programme is supported by the Marie Curie Cofunding of Regional, National and International Programmes (COFUND) of the European Commission.

\end{document}